\def\HAL{1}

\if\HAL 1
\documentclass{article}
\else
\documentclass[graybox]{svmult}
\fi

\usepackage{amsfonts,latexsym,amsmath,amssymb}
\usepackage[mathscr]{eucal}
\usepackage{graphicx,color}
\usepackage{comment}
\usepackage{hyperref}
\if\HAL 1

\newtheorem{proposition}{Proposition}

\newtheorem{definition}{Definition}

\newtheorem{example}{Example}
\fi

\catcode`\@=11
\def\downparenfill{$\m@th\braceld\leaders\vrule\hfill\bracerd$}
\def\overparen#1{\mathop{\vbox{\ialign{##\crcr\crcr
\noalign{\kern0.4ex}
\downparenfill\crcr\noalign{\kern0.4ex\nointerlineskip}
$\hfil\displaystyle{#1}\hfil$\crcr}}}\limits}
\catcode`\@=12

\def\NN{{\mathbb N}}    
\def\RR{{\mathbb R}}    

\def\de{{\widetilde e}}     
\def\dx{{\widetilde x}}     
\def\dX{{\widetilde X}}     
\def\dE{{\widetilde E}}     

\def\dlambda{\widetilde \lambda}
\def\dk{{\widetilde k}}     

\def \dG{\tilde G}

\def\PR{{P}}

\def\der{{\mathfrak{d}}}

\def \factun{\textsf{Fact 1 }}
\def \factde{\textsf{Fact 2 }}
\def \facttr{\textsf{Fact 3 }}

\usepackage{mathptmx}       
\usepackage{helvet}         
\usepackage{courier}        
\usepackage{type1cm}        
\usepackage{makeidx}         
\usepackage{graphicx}        
\usepackage{multicol}        
\usepackage[bottom]{footmisc}

\begin{document}

\if\HAL 1
\title{Lyapunov functions obtained from first order approximations}
\author{V. Andrieu\footnote{
Universit\'e Lyon 1 CNRS UMR 5007 LAGEP, France and Fachbereich C - Mathematik und Naturwissenschaften, Bergische Universit\"at Wuppertal, Germany. 
email{ vincent.andrieu@gmail.com}
}
}
\else
\title*{Lyapunov functions obtained from first order approximations}
\author{Vincent Andrieu}
\institute{Universit\'e Lyon 1 CNRS UMR 5007 LAGEP, France and Fachbereich C - Mathematik und Naturwissenschaften, Bergische Universit\"at Wuppertal, Germany. \email{ vincent.andrieu@gmail.com}}
\fi

\maketitle

\null\hfill \textit{In the honor of the 60{th} birthday of Laurent Praly}
\vspace{4em} 

\abstract{In this paper, we study the construction of Lyapunov functions based on first order approximations. 
In a first part, the study of local exponential stability property of a transverse invariant manifold is considered.
This part is mainly a rephrasing of the result of \cite{AndrieuJayawardhanaPraly_TAC_TransExpStab}.
It is shown with this framework how to construct a Lyapunov function which characterizes this local stability property.
In a second part, when considering the global stability property of an equilibrium point it is shown that the study of first order approximation along solutions of the system allows to construct a Lyapunov function.
}


\noindent{\textbf{Notation~:}}
\begin{itemize}
\item For a vector in $\RR^n$ and a matrix in $\RR^{n\times n}$ the notation $|\cdot|$ stands for the usual $2$ norm.
\item For a positive definite matrix $P$, $\mu_{\max}\{P\}$ and $\mu_{\min}\{P\}$ are respectively the largest and smallest eigenvalue.
\end{itemize}

\section{Introduction}

The use of Lyapunov functions in the study of the stability of solutions or invariant sets of dynamical systems has a long history.
It can be traced back to Lyapunov himself who has introduced this concept in its dissertation  in 1892 (see \cite{Lyapunov_92_IJC_general} for an english translation).
The primary objective of a Lyapunov function is to analyze the behavior of trajectories of a dynamical systems and how this behavior is preserved after perturbations.
However, this tool is also very efficient to synthesize control algorithms as for instance stabilizing control laws, regulators, asymptotic observers (see for instance \cite{Isidori_Book_89,SepulchreJankovic_Book_97,Khalil_Book_02,Praly_Poly_08}).

This is why the study of converse Lyapunov theorem have received a huge attention from the nonlinear control community.
One of the first major contribution to the problem of existence of a Lyapunov function can be attributed to Massera \cite{Massera_AnnalMath_49}.
This results have then subsequently improved over the years (see \cite{Massera_AnnalMath_56,Kurzweil_AMST_56}) and we can quote Teel and Praly who established a theorem of existence of a Lyapunov function in a very general framework in \cite{TeelPraly_ESAIM_00}.
However, despite the rise of a very complete theory to infer the existence of a Lyapunov function, its construction in practice appears to be a very difficult task.

On another hand, using a first order approximation to analyze the local stability of a nonlinear system is the most commonly used approach. 
Indeed, a first order analysis deals intrinsically with linear systems tools and it provides a simple way to construct local Lyapunov functions for a nonlinear system. 

In this note, the \textit{linearization approach} is extended in two directions. The first extension  is the case in which the stability studied concerns a simple manifold and not an equilibrium. This extension has already been published in \cite{AndrieuJayawardhanaPraly_CDC_13}  and \cite{AndrieuJayawardhanaPraly_TAC_TransExpStab} and in this note we briefly rephrase these results.
The second extension is to show that when dealing with equilibrium points, global property may be characterized from first order approximations \textit{along solutions}.
%

In order to introduce these results and aiming at allowing to get a full grip on the key points of the approach 
the following simpler framework is first considered.
Hence, some very classical results are rephrased in the following paragraph.


Consider a nonlinear dynamical system defined on $\RR^{n_e}$ with the origin as equilibrium~:
\begin{equation}\label{eq_SystClass}
\dot e = F(e)\ ,\ F(0)=0\ ,
\end{equation}
with state $e$ in $\RR^{n_e}$ and with a $C^1$ vector field $F:\RR^{n_e}\rightarrow\RR^{n_e}$.
Solutions initiated from $e$ in $\RR^{n_e}$ evaluated at time $t$ are denoted $E(e,t)$.

The origin of system (\ref{eq_SystClass}) is said to be Locally Exponentially Stable (LES for short) if there exist three positive real numbers $k$,  $\lambda$ and $r$ such that the following estimate holds~:
\begin{equation}\label{eq_LES}
|E(e,t)| \leq k \exp(-\lambda t) |e| \ ,\ \forall (e,t)\in \RR^{n_e}\times\RR_+\ ,\ |e|\leq r\ .
\end{equation}

As it is well known, the LES property of the system (\ref{eq_SystClass}) can be checked from the study of the first order approximation  around $"0"$.
Indeed, it is well known (see \cite[Theorem 4.15, p.165]{Khalil_Book_02}) that LES  of the origin of (\ref{eq_SystClass}) is equivalent with exponential stability of the origin of the  linear dynamical system  defined in $\RR^{n_e}$ as follows~:
\begin{equation}\label{eq_LinSystClass}
\dot \de = \frac{\partial F}{\partial e}(0) \, \de\ .
\end{equation}


Constructing a Lyapunov function for the linear system (\ref{eq_LinSystClass}) is an easy task.
Indeed, if the matrix $\frac{\partial F}{\partial e}(0)$ is Hurwitz, and given a positive definite matrix $Q$ in $\RR^{n_e\times n_e}$
the matrix $P$ in $\RR^{n_e\times n_e}$ defined as~:
\begin{equation}\label{eq_PClass}
P = \int_0^{+\infty} \exp\left (\frac{\partial F}{\partial e}(0)s\right )^\top Q \exp\left (\frac{\partial F}{\partial e}(0)s\right ) ds
\end{equation}
is well defined, positive definite and satisfies the Lyapunov algebraic equality~:
\begin{equation}\label{eq_LyapClass}
\frac{\partial F}{\partial e}(0)^\top P + P\frac{\partial F}{\partial e}(0) = - Q\ .
\end{equation}
The former equation implies that the mapping $\de\mapsto\de^\top P \de$ is a Lyapunov function for the  system (\ref{eq_LinSystClass}) since it yields along its trajectories $\dot {\overparen{\de^\top P \de}} =-\de^\top Q \de$.

Moreover, the quadratic function $V(e) = e^\top P e$ is a Lyapunov function for the nonlinear system (\ref{eq_SystClass}) since along its trajectories the following equality holds~:
$$
\dot {\overparen{e^\top P e}} =  2e^\top PF(e) = e^\top \Big [-Q  + 2\int_0^1 P
\underbrace{\left [\frac{\partial F}{\partial e}(se) - \frac{\partial F}{\partial e}(0)\right ]}_{\text{small if $|e|$ small}}ds \Big ]e\ .
$$
This implies that  there exists $r>0$ and $\lambda>0$ such that for all $e$ such that $|e|\leq r $, $\dot {\overparen{e^\top P e}} <-\lambda e^\top P e$. 
This characterizes local exponential stability of the origin of (\ref{eq_SystClass}).


In conclusion to this rephrasing of the simplest framework, the following assertions have been obtained.
\begin{list}{}{%
\parskip 0pt plus 0pt minus 0pt%
\topsep 1ex plus 0pt minus 0pt%
\parsep 0.5ex plus 0pt minus 0pt%
\partopsep 0pt plus 0pt minus 0pt%
\itemsep 1ex plus 0pt minus 0pt
\settowidth{\labelwidth}{1em}%
\setlength{\labelsep}{0.5em}%
\setlength{\leftmargin}{\labelwidth}%
\addtolength{\leftmargin}{\labelsep}%
}
\item[\textbf{\textsf{Fact 1~:}}]
 The exponential stability property for the nonlinear system implies an exponential stability property for the linearized system.

\item[\textbf{\textsf{Fact 2~:}}] The exponential stability property for the linearized system can be characterized by a quadratic Lyapunov function.

\item[\textbf{\textsf{Fact 3~:}}] The Lyapunov function associated to the linearized system may be used directly on the nonlinear system to characterize its stability property.
\end{list}

In the first part of this paper, based on the result of \cite{AndrieuJayawardhanaPraly_TAC_TransExpStab}, we will show that this is also the case when considering  exponential stability of a simple invariant manifold. This allows to introduce Lyapunov function that characterizes the local exponential stability property of an invariant manifold. 

The second part of the paper is devoted to global properties. It will be shown that these three facts are also true when considering the global attractivity of an equilibrium.
Finally, in the conclusion, we introduce some difficulties we are facing when considering the case of the global stability property of an invariant manifold.
This gives a gimps of the results obtained in \cite{AndrieuJayawardhanaPraly_GlobTransExpStab}.

\section{Local transverse exponential stability of a manifold}


\subsection{Transverse local uniform exponential stability }
Throughout this section, instead of considering the system (\ref{eq_SystClass}), a system in the following form is considered.
\begin{equation}
\label{eq_SystemTrans}
\dot e = F(e,x)\ ,\ \dot x = G(e,x)\ ,\ F(0,x)=0\ ,
\end{equation}
where $e$ is in $\RR^{n_e}$, $x$ is in $\RR^{n_x}$
and the functions
$F:\RR^{n_e}\times\RR^{n_x}\rightarrow \RR^{n_e}$ and
$G:\RR^{n_e}\times\RR^{n_x}\rightarrow \RR^{n_x}$ are
$C^2$. We denote by $(E(e,x,t),X(
e,x,t))$ the (unique)
solution which goes through $(e,x)$ in $\RR^{n_e}\times\RR^{n_x}$
at time $t=0$. 
It is assumed that these solutions are defined for all positive times, i.e. the
system is {\it forward complete}.

For this system, the manifold $\mathcal E = \{(e,x), e=0\}$ is an invariant manifold.
The purpose of this section is to show that the properties obtained to characterize the exponential stability property of an equilibrium given in the introduction (i.e. the facts 1, 2 and 3) are still valid when considering the stability property of this manifold.

The local exponential stability of an equilibrium becomes the local exponential stability of the transverse manifold.
This one is defined as follows.
\begin{definition}[Transversal uniform local exponential stability (TULES-NL)]
The system (\ref{eq_SystemTrans}) is forward complete and there exist strictly positive real numbers $r$, $k$ and
$\lambda$ such
that we have, for all $(e_0,x_0,t)$ in $\RR^{n_e}\times\RR^{n_x}\times
\RR_{\geq 0}$ with $|e|\leq r$,
\begin{equation}
\label{eq_ExpStab}
|E(e_0,x_0,t)| \leq k |e_0| \exp(-\lambda t)
\  .
\end{equation}
\end{definition}

In other words, the system (\ref{eq_SystemTrans}) is said to be TULES-NL if the manifold $\mathcal{E}:=\{(e,x):\,  e=0\}$ is
exponentially stable for the system
(\ref{eq_SystemTrans}), locally in $e$
and uniformly in $x$.


\subsection{Fact 1 : Exponential stability of a linearized system }

As mentioned in the introduction, a linearized system "around" the invariant manifold has first to be considered.
In this case, the system is defined as~:
\begin{equation}
\label{eq_System_dif}
\dot \de = \frac{\partial F}{\partial e}(x)\de\ ,\ \dot x = \dG(x)\ ,
\end{equation}
where $\dG(x) = G(0,x)$.

If one wish to show that \textsf{Fact 1} also holds in this context we need to establish that the manifold $\tilde{\mathcal{E}}:=\{(x,\de):\,  \de=0\}$ of the linearized system transversal to $\mathcal E$ in (\ref{eq_System_dif}) is exponentially stable.
This is indeed the case has shown by the following proposition which has been proved in \cite{AndrieuJayawardhanaPraly_TAC_TransExpStab}.
\begin{proposition}[ \cite{AndrieuJayawardhanaPraly_TAC_TransExpStab} \textsf{FACT 1} holds]
\label{Prop_DetecNec}
 If   
Property \textsf{TULES-NL} holds and there exist positive
real numbers $\rho $,
$\mu$
  and $c$ such that, for all $x$ in
$\RR^{n_x}$,
\begin{equation}
\label{eq_BoundLocTrans1}
 \left|\frac{\partial F}{\partial e}(0,x)\right|\leq \mu
\ ,\  
\left|\frac{\partial G}{\partial x}(0,x)\right|\leq \rho
\end{equation}
and, for all $(e,x)$ in $B_e(kr)\times\RR^{n_x}$,
\begin{equation}
\label{eq_BoundLocTrans2}
\left|\frac{\partial^2 F}{\partial e\partial e}(e,x)\right|\leq c\ ,\
\left|\frac{\partial^2 F}{\partial x\partial e}(e,x)\right|\leq c\  ,\
\left|\frac{\partial G}{\partial e}(e,x)\right|\leq c
\; ,
\end{equation}
then the system (\ref{eq_System_dif})
is forward complete and there exist strictly positive real numbers $\dk $  and $\tilde \lambda$
such that any solution $(\dE(\de_0,x_0,t),X(x_0,t))$  of the transversally linear system (\ref{eq_System_dif})
satisfies, for all $(\de_0,x_0,t)$ in $\RR^{n_e}\times\RR^{n_x}\times
\RR_{\geq 0}$,
\begin{equation}\label{eq_ExpStabDriftLocalTrans}
|\dE(\de_0,x,t)|\leq  \dk \exp(-\tilde \lambda t)|\de_0|
\ .
\end{equation}
\end{proposition}
\vspace{0.2cm}

The proof of this proposition   given   in \cite{AndrieuJayawardhanaPraly_TAC_TransExpStab}
 is
 based on the comparison between a given $e$-component of a solution
$\dE(\de_0,x_0,t)$ of (\ref{eq_System_dif}) with pieces of 
$e$-component of solutions    $E(\de_i,x_i,t-t_i)$ of solutions of
(\ref{eq_SystemTrans}) where $\de_i,x_i$ are sequences of points defined on $\dE(\de_0,x_0,t)$.  
Thanks to  the bounds (\ref{eq_BoundLocTrans1}) and (\ref{eq_BoundLocTrans2}),
  it is possible to show that $\dE$ and $E$ remain sufficiently closed so that $\dE$ inherit the convergence property of the solution $E$.
As a consequence, in
  the particular case in which $F$ does not depend on $x$, the two functions $E$ and $\dE$ 
do not
 depend on $x$
either
  and the bounds on the derivatives of the $G$ function 
are useless.

In \cite{AndrieuJayawardhanaPraly_TAC_TransExpStab}, the exponential stability of the manifold $\tilde{\mathcal{E}}:=\{(x,\de):\,  \de=0\}$ of the linearized system transversal to $\mathcal E$ in (\ref{eq_System_dif}) is named property UES-TL.

\subsection{Lyapunov matrix inequality}

The $\de$ components of the system (\ref{eq_System_dif}) is a parametrized time varying linear system.
Hence, the solutions $\dE(e,x,t)$, can be written as~:
$$
\dE(\de,x,t)= \Phi(x,t)\de\ ,
$$
where $\Phi$ is the transition matrix defined as a solution to the following $\RR^{n_e\times n_e}$ dynamical system~:
$$
\dot{\overparen{\Phi(x,t)}} = \frac{\partial F}{\partial e}(0,\dX(\dx,t))\Phi(\dx,t)\ ,\ \Phi(\dx,0)=I\ .
$$
An important point that has to be noticed is that due to equation (\ref{eq_ExpStabDriftLocalTrans}), each element of the (matrix) time function $t\mapsto \Phi(x,t)$ is in $L^2([0,+\infty))$.
Consequently, for all positive definite matrix $Q$ in $\RR^{n_e}$, the matrix function
\begin{equation}\label{eq_PTrans}
P(x) = \lim_{T\rightarrow +\infty } \int_0^T \Phi(x,s)^\top Q \Phi(x,s) ds
\end{equation}
is well defined.

By computing the Lie derivative of the matrix $P$ given in (\ref{eq_PTrans}), it is possible to show that this one satisfies a particular partial differential equation which shows that this function may be used to construct a quadratic Lyapunov function of the linearized system.
 \begin{proposition}[ \cite{AndrieuJayawardhanaPraly_TAC_TransExpStab}   \textsf{FACT 2} holds]
\label{Prop_ExistTensor}
Assume   
Property \textsf{UES-TL} holds, i.e. there exist
 $\dk $  and $\tilde \lambda$
such that any solution $(\dE(\de_0,x_0,t),X(x_0,t))$  of the transversally linear system (\ref{eq_System_dif})
satisfies, \ref{eq_ExpStabDriftLocalTrans}.
Assume moreover, that there exists a positive real number $\mu$ such that
\begin{equation}
\label{LP5}
\left|\frac{\partial F}{\partial e}(0,x)\right|\leq \mu
\qquad \forall x\in \RR^{n_x}\ ,
\end{equation} 
then for all positive definite matrix $Q$, there exists a continuous
function $\PR :\RR^{n_x}\rightarrow\RR^{n_e\times n_e}$
and strictly positive real numbers $\underline{p}$ and $\overline{p}$
such that $\PR $ has a derivative $\der_{\dG}\PR$ along $\dG$ in the
following sense
\begin{equation}
\label{LP9}
\der_{\dG} \PR(\dx)\; := \;
 \lim_{h\to 0}
\frac{\PR(\dX(\dx,h))-\PR(\dx)}{h}\ ,
\end{equation}
and we have, for all $\dx$ in $\RR^{n_x}$,
\begin{eqnarray}
\label{eq_TensorDerivativeTrans}
&\displaystyle
\hskip -3em
\der_{\dG} \PR(\dx) +
\PR (\dx ) \frac{\partial F}{\partial e}(0,\dx)
+ \frac{\partial F}{\partial e}(0,\dx)^\prime \PR(\dx)
\leq -Q\ ,
\\[0.5em]
\label{LP7}
&\displaystyle
\underline{p} \,  I \leq  \PR (\dx)\leq \overline{p}\,  I
\  .
\end{eqnarray}
\end{proposition}
\vspace{0.2cm}

When looking at the time derivative of the function $(\de,x)\mapsto \de^\top P(x) \de$ along the solution of the system (\ref{eq_System_dif}), it yields~:
$$
\dot {\overparen{\de^\top P(x) \de }} = -\de^\top Q \de\ .
$$
Hence, $(\de,x)\mapsto \de^\top P(x) \de$ is a Lyapunov function associated to the $\de$ component of the linearized system (\ref{eq_System_dif}). In other words, 
\factde
 introduced in the introduction is still valid when considering transverse exponential stability property.

The assumption (\ref{LP5}) 
is used to show that $P$ satisfies the left inequality
 in (\ref{LP7}).
Nevertheless this inequality holds without (\ref{LP5}) provided the 
function
 $s\mapsto \left|\frac{\partial \dE}{\partial \de}(0,\dx,s)\right|$ 
does
not
  go too fast to zero.


\subsection{Construction of a Lyapunov function}
From the matrix function $P$ obtained previously, it is possible to define a Lyapunov function which allows to characterize the property of local exponential stability of $\mathcal E$.

\begin{proposition}[ \cite{AndrieuJayawardhanaPraly_TAC_TransExpStab} \textsf{FACT 3} holds]
\label{Prop_Lyap}
If Property \textsf{ULMTE} holds and there exist positive
real numbers $\eta  $ and $c$ such that,
for all $(e,x)$ in $B_e(\eta )\times\RR^{n_x}$,
\begin{eqnarray}
\label{LP8}
&\displaystyle \left| \frac{\partial P}{\partial x} (x)\right|\leq c
\; ,
\\[0.3em]
\label{LP6}
&\hskip -1.6 em
\displaystyle \left|\frac{\partial^2 F}{\partial e\partial
e}(e,x)\right|\leq c\; ,\
\left|\frac{\partial^2 F}{\partial x\partial e}(e,x)\right|\leq c\;  ,\
\left|\frac{\partial G}{\partial e}(e,x)\right|\leq c
\, ,\null
\end{eqnarray}
then Property \textsf{TULES-NL} holds.
\end{proposition}
\vspace{0.2cm}

This is a direct consequence of the use of $V(e,x) = e^\prime P(x)e$ as a Lyapunov function.
The bounds (\ref{LP8}) and (\ref{LP6}) are 
used
 to show that, with equation (\ref{eq_TensorDerivativeTrans}),  the time derivative of this Lyapunov function is negative in a (uniform) tubular neighborhood of the manifold $\{(e,x),e=0\}$.


In conclusion, from Proposition \ref{Prop_DetecNec}, \ref{Prop_ExistTensor} and \ref{Prop_Lyap} it yields that \factun , \factde \,
and \facttr \, 
obtained  in the analysis of local exponential stability of an equilibrium are still valid in the context of local exponential stability of a transverse manifold.
In \cite{AndrieuJayawardhanaPraly_TAC_TransExpStab} the previous framework has been employed as a design tool in different contexts~:
\begin{itemize}
\item It has been employed to construct a Lyapunov function which characterize the property of exponential incremental stability.
\item It has been used to show that a detectability property introduced in \cite{SanfelicePraly_TAC_12} is a necessary condition to the existence of an exponential full order observer.
\item It has been employed (see also in \cite{AndrieuJayawardhanaTarbouriech_CDC_Synchro}) to give necessary and sufficient condition to achieve synchronization.
\end{itemize}

All results written so far concerns local properties. The following section is concerned with global property of an equilibrium point.
Our aim is to follow the same strategy in order to construct global Lyapunov functions.

\section{Global stability properties}

\subsection{Local exponential stability and global attractivity}

In the previous section has been studied the case of the local asymptotic stability (of a manifold or of an equilibrium point).
In this Section, another property is studied : \textit{the global attractivity}.
In other words, we consider again system (\ref{eq_SystClass}) and we assume that for all $e$ in $\RR^{n_e}$,
\begin{equation}
\lim_{t\rightarrow +\infty} |E(e,t)| =0\ .
\end{equation}
Note that global attractivity in combination with the local asymptotic stability of the origin implies that the system is globally and asymptotically stable.
However, it is not globally exponentially stable in the usual sense (see \cite[definition 4.5 p.150]{Khalil_Book_02}) . 
Nevertheless the following property can be simply obtained.
\begin{proposition}
Assume the origin of (\ref{eq_SystClass}) is locally exponentially stable and globally attractive, then there exist a positive real number $\lambda$ and a continuous strictly increasing function $k:\RR_+\rightarrow\RR_+$ such that~:
\begin{equation}\label{eq_LESGA}
|E(e,t)| \leq k(|e|) \exp(-\lambda t) |e|\ .
\end{equation}
\end{proposition}
\begin{proof}
The origin being locally exponentially stable, there exist three positive real numbers $\lambda_1$, $k_1$ and $r_1$ such that equality (\ref{eq_LES}) holds. 
Consider the mapping $c:\RR^{n_e}\times\RR\rightarrow \RR_+$ defined by~:
$$
c(e,t) = \frac{|E(e,t)|}{|e| \exp(-\lambda t)}\ ,
$$
where $0<\lambda<\lambda_1$.
Since inequality (\ref{eq_LES}) holds, it yields that this is a continuous function. Moreover, the global attractivity property and the LES of the origin of system (\ref{eq_SystClass}) implies that~:
$$
\lim_{t\rightarrow +\infty} c(e,t) = 0\ ,\ \forall e\in \RR^{n_e}\ .
$$
Consider the function $\bar c:[r_1,+\infty)\rightarrow\RR_+\cup \{+\infty\}$ defined as~:
$$
\bar c(s) = \sup_{r_1\leq |e|\leq s, t\geq 0} \{ c(e,t) \}\ .
$$
We first show that in fact this function takes finite value for all $s$. Indeed, assume this is not the case for a given $s$, i.e. $\bar c(s) =+\infty$. 
This implies that there exists a sequence $(e_i,t_i)_{i\in\NN}$ with $r_1\leq |e_i|\leq s$ such that 
$c(e_i,t_i)\geq i$. However, $(e_i)_{i\in\NN}$ being a sequence in a compact set, it is possible to extract a sub-sequence $(e_{i_j})_{j\in\NN}$ such that 
$e_{i_j}\rightarrow e^*$ with $r_1\leq |e^*|\leq s$. Note that this implies $t_{i_j}\rightarrow+\infty$. Moreover, by the global attractivity property, there exists $t^*$ such that $|E(e^*,t^*)|\leq \frac{r_1}{2}$. By continuity of the solutions
it yields that there exists $j^*$ such that $|E(e_{i_j},t^*)|\leq r_1$  for $j>j^*$. Without loss of generality, we may assume that $t_{i_j}\geq t^*$ for $j>j^*$. The LES property implies for all $j>j^*$~:
$$
i_j<c(e_{i_j}, t_{i_j}) = \frac{|E(e_{i_j},t_{i_j})|}{|e_{i_j}|\exp(-\lambda t)}\leq \frac{k_1\exp{(-\lambda_1 (t_{i_j}-t^*))}|E(e_{i_j},t^*)|}{|e_{i_j}|\exp(-\lambda t_{i_j})}
\leq \frac{ks\exp(\lambda_1t^*)}{r_1}\ .
$$
Hence a contradiction. Consequently, for all $s\geq r_1$, $\bar c(s)$ is bounded. It is also increasing. So it is possible to select $k:\RR_+\rightarrow\RR_+$ as any continuous function such that~:
$$
k(s)  \geq \left\{
\begin{array}{ll}
k_1 & s\leq r_1\\
\bar c(s) & s\geq r_1\end{array}\right.\ .
$$
It is clear from its definition that the property (\ref{eq_LESGA}) is satisfied.
\end{proof}

\begin{example}
A very simple example of such a property is the scalar system~:
\begin{equation}\label{eq_ExplGlobClass}
\dot e = -\frac{e}{1+e^2}\ .
\end{equation}
Solutions of this ordinary differential equation satisfies the following equations~:
$$
E(e,t) ^2 \exp\left(E(e,t) ^2  \right ) = e^2\exp( e^2)\exp(-2t)\ ,\ \forall e\in \RR\ .
$$
This implies
$$
E(e,t)^2\leq E(e,t)^2 \exp\left(E(e,t) ^2  \right ) \leq e^2\exp( e^2)\exp(-2t)\ ,
$$
and global attractivity and LES of the origin of (\ref{eq_ExplGlobClass}) hold since equation (\ref{eq_LESGA}) is satisfied with 
$k(s)=s\exp\left(\frac 12 s^2\right)$ and $\lambda = 1$.
\end{example}
\subsection{Global Lyapunov functions based on first order approximations}
\subsubsection{Fact 1 : Stability property of the linearized system along the solutions}
A natural question is to know if the local exponential stability and  global attractivity property can be characterized from a first order approximation analysis.
To oppose to the local study made in the introduction, the linearized system around the equilibrium can't describe the property of solutions away from the origin.
Hence, the linearized system along all solutions have to be considered.

Assuming that $F$ is $C^1$ everywhere, the linearized system along trajectories is  defined as~:
\begin{equation}\label{eq_SystClassLinGlob}
\dot \de = \frac{\partial F}{\partial e}(e)\de \ ,\ \dot e = F(e)\ ,
\end{equation}
with $(e,\de)$ in $\RR^{n_e}\times\RR^{n_e}$.
This system is also called the \textit{lifted} system in \cite{GauthierKupka_Book_01} or the \textit{variational system} in \cite{ForniSepulchreVFSchaft_13_CDC_differentialPass}.

Note that the $\de$-components of this system may be rewritten as the following.
\begin{equation}\label{eq_LinGlobrewrite}
 \underbrace{\dot \de= \frac{\partial F}{\partial e}(0)\de}_{\text{(LES)}\Rightarrow \text{goes exp. to zero}} + 
\underbrace{\left[\frac{\partial F}{\partial e}(e)-\frac{\partial F}{\partial e}(0)\right ]}_{\text{(Glob. Attract.)}\Rightarrow \text{ goes to zero}}\de
\end{equation}
The following proposition shows that if the $e$ components go exponentially to zero, then the $\de$ components do the same. 
\begin{proposition}[\textsf{FACT 1} for global property]
Let $F$ be $C^1$ in $\RR^{n_e}$ and $C^2$ around the origin.
Assume the origin of (\ref{eq_SystClass}) is locally exponentially stable and globally attractive, then there exist a positive real number $\dlambda$ and a strictly increasing function $\dk:\RR_+\rightarrow\RR_+$ such that~:
\begin{equation}\label{eq_GlobLinStab}
|\dE(e,t)| \leq \dk(|e|) \exp(-\dlambda t) |\de|\ .
\end{equation}
\end{proposition}
\begin{proof}
The origin being locally exponentially stable, we can define the matrix $P$ as in (\ref{eq_PClass}).
With the algebraic Lyapunov equation (see (\ref{eq_LyapClass})), it yields that along the solution of the system (\ref{eq_SystClass}), the following equality holds~:
\begin{align}
\nonumber\dot{\overparen{\de ^\top P \de}} &= -\de^\top Q \de + 2\de^\top P \left [ \frac{\partial F}{\partial e}(e) - \frac{\partial F}{\partial e}(0)\right ]\de\ ,\\
\label{eq_LyapGlob}&\leq \left [-\frac{\mu_{\min}\{Q\}}{\mu_{\max}\{P\}}  +  \gamma(e)\right ]\de^\top P \de \ ,
\end{align}
where  $\gamma:\RR^{n_e}\rightarrow\RR_+$ is the continuous function defined as
$$
\gamma(e) =2\frac{ \mu_{\max}\{P\}}{\mu_{\min}\{P\}} \left | \frac{\partial F}{\partial e}(e) - \frac{\partial F}{\partial e}(0) \right | ^2 \ .
$$
The function $F$ being $C^2$ around the origin, $\gamma$ is locally Lipschitz around the origin. Hence, there exist two positive real number $r$ and $L$ such that 
\begin{equation}\label{eq_Lipscitzgamma}
\gamma(e) \leq L|e|\ , \ \forall |e|\leq r\ .
\end{equation}
From Gr\"{o}nwall lemma, equation (\ref{eq_LyapGlob}) implies~:
\begin{align}
\nonumber|\dE(\de,e,t)|&\leq \sqrt{\frac{\dE(\de, e,t)^\top  P \dE(\de, e,t)}{\mu_{\min\{P\}}}} \ ,\\
	&\leq \sqrt{\frac{\mu_{\max}\{P\}}{\mu_{\min}\{P\}}}\exp\left (\frac{1}{2}\int_0^t \gamma(E(e,s))ds\right ) \exp\left (-\frac{\mu_{\min}\{Q\}}{2\mu_{\max}\{P\}}t\right ) |\de|\ .\label{eq_Gronwall}
\end{align}
Let $t^*$ be the continuous function defined as~:
$$
t^*(e) = \max\left \{0, \frac{-\ln\left (\frac{r}{k(|e|)|e|}\right )}{\lambda}\right \}\ .
$$
Note that if $k(|e|)|e| \leq r$, $t^*(e)=0$. Moreover, if $k(|e|)|e| > r$, $t^*(e)>0$ and in this case
$$
k(|e|)\exp(-\lambda t^*(e))|e| \leq r\ .
$$
Hence, due to the local exponential stability and global attractivity property,  equation (\ref{eq_LESGA}) yields for all $e$, 
$$
 \left |E\left (e,t^*(e)\right )\right |\leq k(|e|)\exp\left ( -\lambda t^*(e)\right )|e|\leq  r\ .
$$
Employing (\ref{eq_LESGA}), once again, and (\ref{eq_Lipscitzgamma}), the following inequalities are obtained for $t\geq t^*(e)$~:
\begin{align*}
\int_0^t \gamma(E(e,s))ds
&\leq \int_0^{t^*(e)} \gamma(E(e,s))ds + 
\int_{t^*(e)}^t\gamma(E(e,s))ds \ ,\\
&\leq \int_0^{t^*(e)} \gamma(E(e,s))ds + 
Lk(|e|)|e|\int_{t^*(e)}^t\exp(-\lambda s) ds \ ,\\
&\leq \int_0^{t^*(e)} \gamma(E(e,s))ds + 
\frac{Lr}{\lambda} 
:= c(e) \ .
\end{align*}
Notice that the previous inequality is also true for $t\leq t^*(e)$.
Consequently, using the previous approximation in equation (\ref{eq_Gronwall}) the proof ends since equation (\ref{eq_GlobLinStab}) is obtained with~:
$$
\dlambda = \frac{\mu_{\min}\{Q\}}{2\mu_{\max}\{P\}}\ ,\ \dk(s) = \sqrt{\frac{\mu_{\max}\{P\}}{\mu_{\min}\{P\}}}\exp\left (\frac{1}{2}\max_{|e|\leq s}c(e)\right ) \ .
$$
\end{proof}

\begin{example}
Going back to the previous  example given in equation (\ref{eq_ExplGlobClass}), the linearized system is given as~:
$$
\dot \de = -\frac{1-e^2}{1+e^2}\, \de\ .
$$
This gives~:
\begin{align*}
\left |\dE(\de,e,t)\right | &= \exp\left (-t+\int_0^{t}\frac{2E(e,s)^2}{1+E(e,s)^2}ds\right )|\de|\ ,\\
&\leq \exp\left (-t+\int_0^{t}2k(|e|)^2\exp(-s)ds\right )|\de|\ ,\\
&\leq \exp\left (-t+2k(|e|)^2(1-\exp(-t))\right )|\de|\ .
\end{align*}
This gives equation (\ref{eq_GlobLinStab}) with~:
$$
\dk(s)  = \exp\left (2k(s)^2\right ) = \exp\Big (2s^2\exp\left(s^2\right)\Big)  \ , \ \dlambda = 1\ .
$$
\end{example}

\subsubsection{Fact 2 : Lyapunov matrix inequality}
By linearity the $\de$ components of the  linearized system (\ref{eq_SystClassLinGlob})  can be written~:
$$
\dE(\de,e,t)= \Phi(e,t)\de\ ,
$$
where $\Phi$ is the transition matrix. This transition matrix is defined as the solution of the following $\RR^{n_e\times n_e}$ dynamical system~:
$$
\dot{\overparen{\Phi(e,t)}} = \frac{\partial F}{\partial e}(E(e,t))\Phi(e,t)\ ,\ \Phi(e,0)=I\ .
$$
An important point that has to be noticed is that due to equation (\ref{eq_GlobLinStab}), each element of the (matrix) time function $t\mapsto \Phi(e,t)$ is in $L^2([0,+\infty))$.
Consequently, for all positive definite matrix $Q$ in $\RR^{n_e\times n_e}$, the matrix function~:
\begin{equation}\label{eq_PClassGlob}
P(e) = \lim_{T\rightarrow +\infty } \int_0^T \Phi(e,s)^\top Q \Phi(e,s) ds\ ,
\end{equation}
is well defined.
Moreover, it can be shown that the following proposition holds.

\begin{proposition}[\textsf{FACT 2} for global property]\label{Prop_ConsP}
Assume that 
there exist function $(k,\dk)$ and  positive real numbers $(\lambda, \dlambda)$ such that (\ref{eq_LESGA}) and (\ref{eq_GlobLinStab}) are satisfied.
Then, the matrix function $P:\RR^{n_e}\rightarrow\RR^{n_e\times n_e}$ defined in (\ref{eq_PClassGlob}) is well defined, continuous,  and there exist a non increasing function $\underbar p$ and a non decreasing function $\bar p$ such that
\begin{equation}\label{eq_BoundPGlob}
0<\underbar p(|e|) I \leq P(e) \leq \bar p(|e|)I\ ,\ \forall\ e\in\RR^{n_e}\ .
\end{equation}
Moreover\footnote{See the notation (\ref{LP9}).},
\begin{equation}\label{eq_dPdtGlob}
\underbrace{\der_{F} P(e) +
P(e)\frac{\partial F}{\partial e}(e)
+ \frac{\partial F}{\partial e}(e)^\top P(e)}_{=L_FP(e)}
\leq -Q\ ,\ \forall\ e\in\RR^{n_e}\ .
\end{equation}
Finally, if the vector field $F$ is $C^3$ then $P$ is $C^2$.
\end{proposition}
\begin{proof}
%
From  (\ref{eq_GlobLinStab}),
for all $(e,t)$ in
$\RR^{n_e}\times\RR_{\geq 0}$~:
$$
\left|\Phi(e,t)\right|\leq \dk(|e|)\exp(-\tilde \lambda t)\ .
$$
This allows us to claim that, for every symmetric positive definite
matrix $Q$, the function $\PR :\RR^{n_e}\rightarrow\RR^{n_e\times n_e}$ given by  (\ref{eq_PClassGlob}) 
is well defined, continuous and satisfies~:
$$
\mu_{\max}\{\PR (e)\}\leq \frac{\dk(|e|) ^2}{2\tilde \lambda }\mu_{\max}\{Q\}=\overline{p}(|e|)
\ ,\  \forall e\in \RR^{n_e}
\  .
$$
On another hand, let $c$ be a continuous mapping which satisfies the following inequality~:
$$
\left |\frac{\partial F}{\partial e}(e)\right | \leq c(|e|)\ .
$$
Morover, for all $(t,v)$ in $(\RR \times\RR^{n_e})$, we have~:
$$
\frac{\partial }{\partial t}
\left(
v^\prime \left[\Phi(e,t)\right]^{-1}
\right)
=
-v^\prime  \left[\Phi(e,t)\right]^{-1}
\frac{\partial F}{\partial e}(E(e,t))
\  .
$$
However since we have by  (\ref{eq_LESGA})~:
$$
\left |\frac{\partial F}{\partial e}(E(e,t))\right |\leq c(|k(e)|\,|e|)\ ,
$$
it yields the following estimate~:
$$
\left|v^\prime \Phi(e,t)^{-1}\right|
\leq \exp\Big(c(k(|e|)|e|) t\Big)
\left|v\right|\ , \ \forall (t,v) \in (\RR \times\RR^{n_e})\ .
$$
This implies for all $(t,v)$ in $(\RR \times\RR^{n_e})$~:
\begin{align*}
[v^\prime v]^2
&\leq 
\left|v^\prime\Phi(e,t)^{-1}\right|^2
 \left|\Phi(e,t) v \right|^2\ ,\\
&\leq 
\frac{1}{\mu_{\min}\{Q\}}
\left|v^\prime\Phi(e,t)^{-1}\right|^2
v^\prime\Phi(e,t)^\prime Q \Phi(e,t)v \ ,
\\
&
\leq
\frac{|v|^2 \exp\Big(2c(k(|e|)|e|) t\Big)}{\mu_{\min}\{Q\}}\: 
v^\prime\Phi(e,t)^\prime Q \Phi(e,t) v \ .
\end{align*}
So, this yields~:
$$
v^\prime\Phi(e,t)^\prime Q \Phi(e,t) v  \geq \mu_{\min}\{Q\} \exp\Big(-c(k(|e|)|e|) t\Big)|v|^2\ , \ \forall (t,v) \in (\RR \times\RR^{n_e})\ .
$$
Consequently, we get~:
$$
\underline{p}(|e|)= \frac{\mu_{\min}\{Q\}}{2 c(k(|e|)|e|)  }
\leq \lambda_{\min}\{\PR (e)\}
\qquad \forall \de\in \RR^{n_e}
\  .
$$

Finally, to get (\ref{eq_dPdtGlob}), let us exploit the semi
group property of the solutions.
We have for all $(\de,e)$ in $\RR^{n_e}\times\RR^{n_e}$ and all
$(t,r)$ in $\RR_{\geq 0}^2$~:
$$
\dE(\dE(\de,e,t),E(e,t),r)=
\dE(\de,e, t+r)\ .
$$
Differentiating with respect to $\de$ the previous equality yields~:
$$
\frac{\partial \dE}{\partial \de}(\dE(\de,e,t),E(e,t),r)
\frac{\partial \dE}{\partial \de}(\de,e,t)
=
\frac{\partial \dE}{\partial \de}(\de,e, t+r)\ .
$$
Hence, we get the property~:
$$
\Phi(E(e,t),r)
\Phi(e,t)
=
\Phi(e, t+r)\ .
$$
Setting in the previous equality~:
$$
e:= E(e,h)\ ,\ h:=-t\ ,\ s:=t+r\ ,
$$
we get for all $e$ in $\RR^{n_e}$ and all $(s,h)$ in $\RR^2$~:
$$
\Phi(e,s+h)
\Phi(E(e,h),-h)
=
\Phi(E(e,h), s)\ .
$$
Consequently, this yields~:
\begin{align*}
\PR (E(e,h))=
&= \lim_{T\to +\infty }\int_0^{T}
\Phi(E(e,h),s)^\prime Q\Phi(E(e,h),s)
ds\ ,
\\
&= \lim_{T\to +\infty }
\left(
\Phi(E(e,h),-h)
\right)^\prime  
\left[\int_0^{T} \left(
\Phi(e,s+h)
\right)^\prime Q\Phi(e,s+h)
ds
\right]\\&\qquad\qquad\qquad\qquad\qquad\qquad\qquad\qquad\qquad
\Phi(E(e,h),-h)\ .
\end{align*}
But we have~:
\begin{eqnarray*}
&\displaystyle
\lim_{h\to 0}\frac{\Phi(E(e,h),-h)-I}{h}
= -\frac{\partial F}{\partial e}(e)
\  ,
\\[0.5em]
&\displaystyle
\lim_{h\to 0}\frac{
\Phi(e,s+h)
-
\Phi(e,s)
}{h}
=
\frac{\partial }{\partial s}\left(\Phi(e,s)\right)\ ,
\end{eqnarray*}
and
\begin{multline*}
\int_0^T
\frac{\partial }{\partial s}\left(\Phi(e,s)\right)^\prime
Q
\left(\Phi(e,s)\right)
ds
+
\int_0^T
\left(\Phi(e,s)\right)^\prime
Q
\frac{\partial }{\partial s}\left(\Phi(e,s)\right)
ds
=\\
\Phi(e,T)^\prime
Q
\Phi(e,T)
-
Q
\ . \end{multline*}
Since $\lim_T$ and $\lim_h$ commute because of the exponential
convergence to $0$ of $\Phi(e,s)$, we conclude that  (\ref{eq_dPdtGlob}) is satisfied.

The last assertion of the proposition is simply obtained noticing that if $F$ is $C^3$ then the matrix function $\Phi(e,t)$ is also $C^2$ in $e$.
Moreover the first and second derivatives of its coefficient belong also to $L^{2}[0,+\infty)$.
\end{proof}

\begin{example}
If we pursue the analysis for the scalar example given in equation (\ref{eq_ExplGlobClass}), it yields
$$
\underline p(e) = \frac{1}{2}\ ,\ \overline p(e) \leq \dk(e)^2\lim_{T\rightarrow +\infty}\int_0^T \exp(-2s)ds = \frac{\exp\Big( 4e^2\exp(e^2) \Big)}{2}\ .
$$
\end{example}

\subsubsection{Fact 3 : Construction of a Lyapunov function}

With the matrix function $P$ defined for instance in (\ref{eq_PClassGlob}) which Lie derivative satisfies inequality  (\ref{eq_dPdtGlob}), it yields that along the solution of the linearized system (\ref{eq_SystClassLinGlob})~:
$$
\dot{\overparen{\de^\top P(e) \de}} = -\de^\top Q \de\ .
$$
In other words, the mapping $(\de,e)\mapsto \de^\top P(e) \de$ is a global Lyapunov function for the $\de$ components of the linearized system (\ref{eq_PClassGlob}).

However,  $e\mapsto e^\top P(e) e$ is not a global Lyapunov function for $\dot e = F(e)$.
Indeed, a simple computation gives~:
$$
\dot{\overparen{e^\top P(e) e}} = 2e^\top P(e) \left [F(e) - \frac{\partial F}{\partial e}(e)e\right ] - e^\top Qe\ .
$$
This is negative definite if $F(e) - \frac{\partial F}{\partial e}(e)e$ is small. However, there is no guarantee that this is case away from the origin.

Nevertheless, it is still possible to construct a Lyapunov function for the system (\ref{eq_SystClass}).
Indeed,  the matrix function $P$   may be used to define a Riemanian metric on $\RR^{n_e}$ which may be used as a Lyapunov function.
Precisely, if $P$ is a $C^2$ 
function  the values of which are symmetric 
matrices satisfying (\ref{eq_BoundPGlob}),
The length of  any piece-wise $C^1$ path $\gamma :[s_1,s_2]\to
\RR^{n_e}$ between two arbitrary points $e_1=\gamma (s_1)$ and 
$e_2=\gamma (s_2)$ in
$\RR^{n_e}$ is defined as~:
\begin{equation}\label{eq_RiemanianLength}
\left. L(\gamma)\vrule height 0.51em depth 0.51em width 0em \right|_{s_1}^{s_2}=\int_{s_1}^{s_2}\sqrt{\frac{d\gamma}{ds}(\sigma )^\prime  P (\gamma(\sigma ))\frac{d\gamma}{ds}(\sigma )}\: d\sigma\ .
\end{equation}
By minimizing along all such path we get the distance  $d_P(e_1,e_2)$.

Then, thanks to
the well established
relation between (geodesically) monotone vector field (semi-group
generator) (operator) and contracting (non-expansive) flow (semi-group)
(see \cite{Lewis_49_AJM_MetPropDiffEq,Hartman_Book_64,Brezis_Book_73,IsacNemeth_Book_08} and many
others), we know that if $P$ is $C^2$ and the metric space is complete, this distance between any two solutions of 
(\ref{eq_SystClass}) is exponentially decreasing to $0$ as time goes on 
forward if (\ref{eq_dPdtGlob}) is satisfied with $Q$ is a positive definite symmetric matrix.
For a proof, see for example \cite[Theorem 1]{Lewis_49_AJM_MetPropDiffEq} or \cite[Theorems
5.7 and 5.33]{IsacNemeth_Book_08}
or \cite[Lemma 3.3]{Reich_Book_05_NLSemGrp} (replacing $f(x)$ by $x+ h
f(x)$).

From this fact, a  candidate Lyapunov function is  the Riemannian distance to the origin.
Hence we introduce the function $V:\RR^{n_e}\rightarrow\RR_+$
\begin{equation}\label{eq_LyapRiem}
V(e) = d_{P}(e,0)\ .
\end{equation}
In the following proposition it is shown that this function is indeed a good Lyapunov function candidate and moreover that it admits an upper Dini derivative along the solution of the system (\ref{eq_SystClass}) which is negative definite.
\begin{proposition}[\textsf{FACT 3} for global property]\label{Prop_LyapSuffGlobStab}
Assume $F$ is $C^2$.
Assume moreover that there exists a $C^2$ matrix function $P$ such that equations (\ref{eq_BoundPGlob}) and (\ref{eq_dPdtGlob}) hold and that the function $\underline{p}$ satisfies the following property
\begin{equation}\label{Ass_boundP}
\lim_{r\rightarrow +\infty}\underline p(r)r^2 = +\infty\ .
\end{equation}
Then the function $V$ defined in (\ref{eq_LyapRiem}) is a Lyapunov function for the system (\ref{eq_SystClass}). 
More precisely $V$ admits an upper Dini derivative along the solutions of system (\ref{eq_SystClass}) defined as
$$
D^+_F V(e) := \limsup_{h\searrow 0} \frac{V(E(e,h))-V(e)}{h}\ ,
$$
which satisfies
$$
D^+_F V(e)  \leq -\frac{\mu_{\min}\{Q\}}{\bar p(|e|)}V(e)\ .
$$
Hence the origin is locally exponentially stable and globally attractive.
\end{proposition}
\begin{proof}
Given an initial point $e$ in $\RR^{n_e}$ and a direction $v$ also in $\RR^{n_e}$, geodesics are given as solution to the geodesic equation:
\begin{equation}\label{eq_GeodEq}
\frac{d^2 \gamma_\ell}{ds^2}(s)(s) = \sum_{i,j}^{n}\mathfrak\Gamma_{ij}^\ell \frac{d\gamma_i}{ds}(s)\frac{d\gamma_j}{ds}(s)\ ,\ \gamma(0)=e\ ,\ 
\frac{d\gamma}{ds}(0)=v\ ,
\end{equation}
where the $(\mathfrak\Gamma_{ij}^\ell)$ are
Christoffel
symbols associated to $P$ which are $C^1$ if $P$ is $C^2$.
The right hand side of the previous equation being $C^1$ we know that solutions $\left(\gamma(s),\frac{d\gamma}{ds}(s)\right)$ of (\ref{eq_GeodEq}) exist at least for small $s$, are unique and $C^1$. Hence, $\gamma(\cdot)$ is $C^2$ on its domain of existence.

Now, with \cite[Lemma A.1]{SanfelicePraly_TAC_12} and the assumption given in equation (\ref{Ass_boundP}) it yields that these geodesic can be maximally extended to $\RR$.
With Hopf-Rinow Theorem, this implies that the metric space $(\RR^{n_e},P)$ is complete.
Moreover, for any $e$ in $\RR^{n_e}$ there exists $\gamma^*:[0,s_e]\rightarrow \RR^{n_e}$ a $C^2$ curve (a geodesic) such that~:
$$
d_P(e,0) = \left. L(\gamma^*)\vrule height 0.51em depth 0.51em width 0em \right|_{0}^{s_e}\ .
$$

As a convention, it is assumed in the following and without loss of generality that the geodesics are normalized~:
$$
\frac{d\gamma^*}{ds}(s)^\top P(\gamma^*(s)) \frac{d\gamma^*}{ds}(s) = 1\ .
$$
Hence the function $V$ defined in (\ref{eq_LyapRiem}) satisfies~:
$$
V(e) = \int_0^{s_e} \sqrt{\frac{d\gamma^*}{ds}(s)^\top P(\gamma^*(s)) \frac{d\gamma^*}{ds}(s)} \, ds
= \int_0^{s_e} \frac{d\gamma^*}{ds}(s)^\top P(\gamma^*(s)) \frac{d\gamma^*}{ds}(s) \, ds = s_e\ .
$$

Let us first show that $V$ is a positive definite and proper function.
Since $\gamma^*:[0,s_e]$ is a continuous path from $e$ to zero, this implies that there exists $s_0$ in $[0,s_e]$ such that~:
\begin{equation}\label{eq_geotronc}
|\gamma^*(s_0)| = |e|\ ,\ |\gamma^*(s)|\leq |e|\ ,\ \forall s\in [s_0, s_e]\ .
\end{equation}
Note that~:
\begin{align*}
V(e) & = \int_0^{s_e} \sqrt{\frac{d\gamma^*}{ds}(s)^\top P(\gamma^*(s))\frac{d\gamma^*}{ds}(s)}ds\ ,\\
& \geq \int_{s_0}^{s_e} \sqrt{\frac{d\gamma^*}{ds}(s)^\top P(\gamma^*(s))\frac{d\gamma^*}{ds}(s)}ds\ ,\\
& \geq \sqrt{\underline p(|e|)}\int_{s_0}^{s_e} \sqrt{\frac{d\gamma^*}{ds}(s)^\top \frac{d\gamma^*}{ds}(s)}ds\ ,
\end{align*}
Since minimal geodesic for an Euclidean metric are straight lines $s\mapsto \frac{s\gamma^*(s_0)}{s_e-s_0}$, this implies~:
$$
\int_{s_0}^{s_e} \sqrt{\frac{d\gamma^*}{ds}(s)^\top \frac{d\gamma^*}{ds}(s)}ds
\geq \int_{s_0}^{s_e} \sqrt{\frac{\gamma^*(s_0)^\top \gamma^*(s_0)}{(s_e-s_0)^2}}ds=|\gamma^*(s_0)|\ .
$$
Hence, with (\ref{eq_geotronc}), it implies~:
\begin{align*}
V(e)& \geq \sqrt{\underline p(|e|)} |e|
\ .
\end{align*}
Moreover,
\begin{align*}
V(e) & \leq \int_0^{s_e} \frac{e^\top}{s_e} P\left (\frac{se^\top}{s_e}\right )\frac{e}{s_e}ds\ ,\\
& \leq \overline p(|e|) \int_0^{s_e} \frac{e^\top}{s_e} \frac{e}{s_e}ds\ ,\\
& \leq \frac{\overline p(|e|)}{s_e} |e|^2 \ .
\end{align*}
Since we have $V(e)=s_e$, the two previous inequalities imply the following~:
\begin{equation}\label{eq_PropRiemV}
\sqrt{\underline p(|e|)} |e| \leq V(e) \leq \sqrt{\overline p(|e|)} |e|\ .
\end{equation}
With (\ref{Ass_boundP}), this implies that the function $V$ is positive definite and proper.

Let now $\Gamma(s,t)$ be the mapping defined by~:
$$
\frac{\partial \Gamma}{\partial t}(s,t) = F(\Gamma(s,t))\ ,\ \Gamma(s,0) = \gamma^*(s)\ .
$$
The vector field $F$ being $C^2$, the mapping $\gamma^*$ being $C^2$, it yields that $\Gamma$ is $C^2$.
Note that $\Gamma(s,h)$ is a $C^2$ path such that~:
$$
\Gamma(s_e,h) = E(e,h)\ ,\ \Gamma(0,h) = 0\ .
$$
This implies the following inequality for all $h\geq 0$~:
$$
V(E(e,h)) \leq  \int_0^{s_e}\sqrt{ \frac{\partial \Gamma}{\partial s}(s,h)^\top P(\Gamma(s,h)) \frac{\partial \Gamma}{\partial s}(s,h)} ds\ .
$$
This yields~:
$$
D^+_F V(e) \leq \limsup_{h\rightarrow 0} \int_0^{s_e} \frac{\sqrt{\frac{\partial \Gamma}{\partial s}(s,h)^\top P(\Gamma(s,h)) \frac{\partial \Gamma}{\partial s}(s,h)}-\sqrt{\frac{\partial \gamma^*}{\partial s}(s)^\top P(\gamma^*(s))\frac{\partial \gamma^*}{\partial s}(s)} }{h} ds\ .
$$
Hence, with Fatou's lemma, it yields~:
$$
D^+V(e)\leq  \int_0^{s_e} \limsup_{h\rightarrow 0} \frac{\sqrt{\frac{\partial \Gamma}{\partial s}(s,h)^\top P(\Gamma(s,h)) \frac{\partial \Gamma}{\partial s}(s,h)}-\sqrt{\frac{\partial \gamma^*}{\partial s}(s)^\top P(\gamma^*(s))\frac{\partial \gamma^*}{\partial s}(s)} }{h} ds\ .
$$
The mapping $h\mapsto \sqrt{\frac{\partial \Gamma}{\partial s}(s,h)^\top P(\Gamma(s,h)) \frac{\partial \Gamma}{\partial s}(s,h)}$ being $C^1$ (since $\Gamma$ and $P$ are $C^2$), it yields
\begin{align*}
D^+V(e) 
&\leq  \int_0^{s_e} \frac{\partial}{\partial h}\left\{\sqrt{\frac{\partial \Gamma}{\partial s}(s,\cdot)^\top P(\Gamma(s,\cdot)) \frac{\partial \Gamma}{\partial s}(s,\cdot)}\right\}_{h=0} ds\ ,\\
&=-\int_0^{s_e} \frac{1}{2}\frac{\frac{d \gamma^*}{ds}(s)^\top Q \frac{d \gamma^*}{ds}(s)}{\sqrt{\frac{\partial \gamma^*}{\partial s}(s)^\top P(\gamma^*(s)) \frac{\partial \gamma^*}{\partial s}(s)}}ds\ ,\\
&\leq - \frac{1}{2}\mu_{\min}\{Q\}\int_0^{s_e} \frac{d \gamma^*}{ds}(s)^\top \frac{d \gamma^*}{ds}(s) ds\ ,
\end{align*}
where the last inequality employs the fact that the geodesics are normalized.
With Cauchy-Schwartz inequality, this implies~:
\begin{align*}
D^+V(e) 
&\leq - \frac{1}{2}\mu_{\min}\{Q\}\left(\int_0^{s_e} \sqrt{\frac{d \gamma^*}{ds}(s)^\top \frac{d \gamma^*}{ds}(s)} ds\right)^2\ .
\end{align*}
Since minimal geodesic for an Euclidean metric are straight lines, this implies~:
\begin{align*}
D^+V(e) 
&\leq - \frac{1}{2}\mu_{\min}\{Q\} \int_0^{s_e} \sqrt{ \frac{e}{s_e}^\top \frac{e}{s_e}} ds\ ,\\
&\leq - \frac{\mu_{\min}\{Q\}}{2\sqrt{\overline  p(|e|)}}  \int_0^{s_e}  \sqrt{\frac{e}{s_e}^\top P\left (\frac{s e}{s_e} \right )\frac{e}{s_e}} ds\ ,\\
&\leq - \frac{\mu_{\min}\{Q\}}{2\sqrt{\overline  p(|e|)}}  V(e)\ .
\end{align*}
%
This, together with (\ref{eq_PropRiemV}) implies global asymptotic stability of the origin.
Since $0<\underline p(0)<\overline{p}(0)$, it also implies that the origin is locally exponentially stable.
\end{proof}

Note that an interesting property of the considered Lyapunov function is that given two points $e_1$ and $e_2$ both in $\RR^{n_e}$, if we denote $d_P(e_1,e_2)$ the Riemmanian distance between these two points and $\gamma^*$ the minimal (and normalized) geodesic, it yields following the previous proof, this implies that there exists $s_0$ such that 
$$
|\gamma^*(s_0)-e_2| = |e_1-e_2|\ ,\ |\gamma^*(s)-e_2|\leq |e_1-e_2|\ ,\ \forall s\in [s_0, s_2]\ .
$$
From this, it yields
$$
d_P(e_1,e_2)  \geq \int_{s_0}^{s_2} \sqrt{\frac{d \gamma^*}{ds}(s)^\top P(\gamma^*(s))\frac{d \gamma^*}{ds}(s)}ds\ ,
$$
Moreover, for all $s$ in $[s_0, s_2]$, we have $|\gamma^*(s)| \leq |\gamma^*(s)-e_2| + |e_2|\leq |e_1-e_2| + |e_2|$.
Hence, it yields~:
\begin{align*}
d_P(e_1,e_2) 
& \geq\sqrt{ \underline p(|e_1-e_2|+|e_2|)}\int_{s_0}^{s_2} \sqrt{\frac{d \gamma^*}{ds}(s)^\top \frac{d \gamma^*}{ds}(s)ds}\ ,\\
& \geq\sqrt{ \underline p(|e_1-e_2|+|e_2|)}|\gamma^*(s_0)-e_2|
\ ,\\
& \geq\sqrt{ \underline p(|e_1-e_2|+|e_2|)}|e_1-e_2|\ .
\end{align*}
Moreover,
$$
d_P(e_1,e_2)  
\leq \frac{\overline p(|e_1-e_2|+|e_2|))}{d_P(e_1,e_2)} |e_1-e_2|^2\ .
$$
The two previous inequalities imply~:
$$
\sqrt{\underline p(|e_1-e_2|+|e_2|))} |e_1-e_2|\leq d_P(e_1,e_2) \leq \sqrt{\overline p(|e_1-e_2|+|e_2|))} |e_1-e_2|\ .
$$
Moreover, we have
$$
D^+_{F,F} d_P(e_1,e_2) \leq -\int_{s_1}^{s_2}\dot \gamma^*(s) Q \dot \gamma^*(s) ds \leq -\frac{\mu_{\min}\{Q\}}{2\sqrt{\underline p( |e_1-e_2|+|e_2|})} d_p(e_1,e_2)\ ,
$$
where
$$
D^+_{F,F} d_P(e_1,e_2) := \limsup_{h\searrow 0} \frac{d_P(E(e_1,h)),E(e_2,h))}{h}\ .
$$
In other words, there exists a strictly decreasing distance between any two points.
Consequently, it yields exponential convergence of the euclidean distance between any two trajectories toward zero.
Hence, roughly speaking, we have shown that when the origin is locally exponentially stable and globally attractive then there exists a strictly decreasing distance between any two trajectories.
However, this convergence is not uniform in $e_1$ and $e_2$. 
This is a strong difference with the property of incremental stability as studied for instance in \cite{Angeli_TAC_02_lyapIncStab} or \cite{ForniSepculchre_TAC_2014}.
Note moreover that it is shown in \cite{RufferEtAl_SCL_13ConvSystIncStab} that the asymptotic stability property and incremental stability property are different.

Note that as mentioned in \cite{AndrieuJayawardhanaPraly_TAC_TransExpStab}, when the two function $\underline p$ and $\overline p$ are respectively lower and upper bounded by a nonzero constant then the convergence obtained is uniform. In this case, the usual definition of incremental stability is recovered.

\subsubsection{About the requirement (\ref{Ass_boundP}) }

The requirement (\ref{Ass_boundP}) is essential to make sure that $\RR^{n_e}$ endowed with the Riemannian metric $P$ is complete. It is also essential to make sure that the obtained Lyapunov function is proper. It imposes that the mapping $\underline{p}$ doesn't vanish to quickly as $|e|$ goes to infinity. Going back to a definition of the mapping $\underline p$ obtained in the proof of Proposition \ref{Prop_ConsP}, it yields that if the vector field $F$ is globally Lipschitz then $\underline p$ is a constant. In other words, in the globally Lipschitz context this assumption is trivially satisfied.

Another solution to make sure that this assumption is satisfied is to modify the function $P$ to make sure that this one is lower bounded by a positive real number.
Indeed, note that the trajectories of the system
$$
\dot e = \frac{F(e)}{1 + \left|\frac{\partial F}{\partial e}(e)\right|^3} \ ,\ \dot \de = \frac{\frac{\partial F}{\partial e}(e)}{1 + \left|\frac{\partial F}{\partial e}(e)\right|^3}\de
$$
are the same than the one of the lifted system (\ref{eq_SystClassLinGlob}) (this system is obtained after a time rescaling).  
Consequently, the origin is globally attractive.
Moreover, it is not difficult to show that its origin is also locally exponentially stable.
Finally, if $F$ is $C^4$ then the vector field $e\mapsto\frac{F(e)}{1 + \left|\frac{\partial F}{\partial e}(e)\right|^3}$ is $C^3$.
Let $\tilde \Phi$ be the transition matrix defined as the solution of the following $\RR^{n_e\times n_e}$ dynamical system~:
$$
\frac{d}{dt}\tilde \Phi\,(e,t) = \frac{\frac{\partial F}{\partial e}(E(e,t))}{1 + \left|\frac{\partial F}{\partial e}(E(e,t))\right|^3}\tilde \Phi(e,t)\ ,\ \tilde \Phi(e,0)=I\ .
$$
Again, each element of the (matrix) time function $t\mapsto \tilde \Phi(e,t)$ is in $L^2([0,+\infty))$.
Consequently, for all positive definite matrix $Q$ in $\RR^{n_e\times n_e}$, the matrix function~:
\begin{equation}\label{eq_PClassGlobbounded}
\tilde P(e) = \lim_{T\rightarrow +\infty } \int_0^T \tilde\Phi(e,s)^\top Q \tilde \Phi(e,s) ds\ ,
\end{equation}
is well defined.
With this mapping, the following property may be obtained.
\begin{proposition}[Lower bounded $P$]\label{Prop_ConsPbounded}
Assume that 
there exist function $(k,\dk)$ and  positive real numbers $(\lambda, \dlambda)$ such that (\ref{eq_LESGA}) and (\ref{eq_GlobLinStab}) are satisfied.
Then, the matrix function $P:\RR^{n_e}\rightarrow\RR^{n_e\times n_e}$ defined in (\ref{eq_PClassGlobbounded}) is well defined, continuous,  and there exist a a positive real number $\underbar p$ and a non decreasing function $\bar p$ such that
\begin{equation}\label{eq_BoundPGlobbounded}
0<\underbar p  I \leq \tilde P(e) \leq \bar p(|e|)I\ ,\ \forall\ e\in\RR^{n_e}\ .
\end{equation}
Moreover,
\begin{equation}\label{eq_dPdtGlobBd}
\der_{F} \tilde P(e) +
\tilde P(e)\frac{\partial F}{\partial e}(e)
+ \frac{\partial F}{\partial e}(e)^\top \tilde P(e)
\leq -Q\left(1+\left|\frac{\partial F}{\partial e}(e)\right|^3\right)\ ,\ \forall\ e\in\RR^{n_e}\ .
\end{equation}
Finally, if the vector field $F$ is $C^4$ then $P$ is $C^2$.
\end{proposition}
\begin{proof}
The proof follows the same step then the one of Proposition \ref{Prop_ConsP}.
For all $(e,t)$ in
$\RR^{n_e}\times\RR_{\geq 0}$ there exists~:
$$
\left|\tilde \Phi(e,t)\right|\leq \dk(|e|)\exp(-\tilde \lambda t)\ .
$$
This allows us to claim that, for every symmetric positive definite
matrix $Q$, the function $\PR :\RR^{n_x}\rightarrow\RR^{n_e\times n_e}$ given by  (\ref{eq_PClassGlob}) 
is well defined, continuous and satisfies~:
$$
\mu_{\max}\{\PR (e)\}\leq \frac{\dk(|e|) ^2}{2\tilde \lambda }\mu_{\max}\{Q\}=\overline{p}(|e|)
\ ,\  \forall e\in \RR^{n_e}
\  .
$$
Morover, for all $(t,v)$ in $(\RR \times\RR^{n_e})$, we have~:
$$
\frac{\partial }{\partial t}
\left(
v^\prime \left[\Phi(e,t)\right]^{-1}
\right)
=
-v^\prime  \left[\Phi(e,t)\right]^{-1}
\frac{\frac{\partial F}{\partial e}(E(e,t))}{1+\left| \frac{\partial F}{\partial e}(E(e,t))\right|^3}
\  .
$$
However since we have by  (\ref{eq_LESGA})
$$
\left |\frac{\frac{\partial F}{\partial e}(E(e,t))}{1+\left| \frac{\partial F}{\partial e}(E(e,t))\right|^3}\right |\leq 1\ ,
$$
it yields the following estimate~:
$$
\left|v^\prime \Phi(e,t)^{-1}\right|
\leq \exp( t)
\left|v\right|\ , \ \forall (t,v) \in (\RR \times\RR^{n_e})\ ,
$$
From this following the proof of Proposition \ref{Prop_ConsP}, it yields~:
$$
\underline{p} = \frac{\mu_{\min}\{Q\}}{2 }
\leq \lambda_{\min}\{\PR (e)\}
\qquad \forall \de\in \RR^{n_e}
\  .
$$
Also, the following inequality may be obtained~:
$$
\der_{\frac{F}{1+|\frac{\partial F}{\partial e}(e)|^3}} \tilde P(e) +
\tilde P(e)\frac{\frac{\partial F}{\partial e}(e)}{1+|\frac{\partial F}{\partial e}(e)|^3}
+ \frac{\frac{\partial F}{\partial e}(e)^\top}{1+|\frac{\partial F}{\partial e}(e)|^3} \tilde P(e)
\leq -Q\ ,\ \forall\ e\in\RR^{n_e}\ .
$$
Multiplying the former equation by $1+|\frac{\partial F}{\partial e}(e)|^3$ and it yields the result.
\end{proof}

Since the matrix function $P$ is lower bounded, we can define a Lyapunov function following the proposition \ref{Prop_LyapSuffGlobStab}.
Roughtly speaking we have established the following Lyapunov inverse result~:
\textit{Assuming some regularity on the system, if the origin is locally exponentially stable and globally attractive then there exists a strictly decreasing Lyapunov function given as a Riemannian distance to the origin.}

Of course the local exponential stability property is essential.
Note that in \cite{GruneSontagWirth_SCL_99asymptotic}, is shown that up to a change of coordinates (which is not a diffeomorphism since it is not smooth at the origin) it is possible to transform any asymptotically stable system in an exponentially stable system. This implies that up to a change of variable, it is always possible to consider Lyapunov function coming from a Riemannian distance.

\subsection{Stabilization}

From the previous analysis, it has been shown that a linearization approach leads to the constuction of global Lyapunov function in the case of local exponential stability and global attractivity. It may be interesting to know if this type of Lyapunov function may be used in control design.

We consider here a  controlled nonlinear system given on $\RR^n$ as
\begin{equation}\label{eq_SystTrack}
\dot w=f(w)+g(w)u\ ,
\end{equation}
with $f:\RR^n\rightarrow\RR^n$ and $g:\RR^n\rightarrow\RR^n$ are smooth vector fields and $u$ the controlled input is in $\RR$.

Our objective is to construct a control $u=\phi(w)$ that achieves local exponential stabilization and global attractivity of the origin.
Based on the former analysis, a sufficient condition based on the use of a Riemanian Lyapunov function may be given.
Note however that these assumptions inspired from \cite{ForniSepulchreVFSchaft_13_CDC_differentialPass} and \cite{AndrieuJayawardhanaPraly_CDC_13} are very conservative.
\begin{proposition}
Assume there exists a mapping $P:\RR^n\rightarrow\RR^{n\times n}$ such that
\begin{enumerate}
\item The matrix function $P$ is $C^3$, satisfies the condition 
 (\ref{eq_BoundPGlob}),  (\ref{Ass_boundP}) and there exists a positive real number $\lambda$ and a positive definite matrix $Q$ such that the following matrix inequality holds~:
 \begin{equation}\label{eq_dPdtGlobStabilization}
\der_{f} P(w) +
P(w)\frac{\partial f}{\partial w}(w)
+ \frac{\partial f}{\partial w}(w)^\top P(w) - \lambda\left | P(w)g(w) \right |^2
\leq -Q\ ,\ \forall\ w\in\RR^{n}\ .
\end{equation}
\item $g$ is a Killing vector field for the metric $P$. In other word, for all $w$ in $\RR^n$~:
$$
L_gP(w) = \der_{g} P(w) +
P(w)\frac{\partial g}{\partial w}(w)
+ \frac{\partial g}{\partial w}(w)^\top P(w) = 0\ .
$$
\item  There exists a mapping $U:\RR^n\rightarrow\RR$ such that~:
\begin{equation}\label{eq_IntegCond}
\frac{\partial U }{\partial  w}(w) = P(w)g(w)^\top\ .
\end{equation}
\end{enumerate}
Then the control law $u = -\lambda U(w)$ achieves  local exponential stability and globally attractivity of the origin of system (\ref{eq_SystTrack}) in closed loop. 
\end{proposition}
\begin{proof}
The proof follows readily from Proposition \ref{Prop_LyapSuffGlobStab}. Indeed, the closed loop system may be rewritten as
$$
F(w) = f(w)-\lambda g(w) U(w)\ .
$$
Note that if we compute the Lie derivative of the tensor $P$, it yields
$$
L_FP(w) = L_f P(w) - \lambda L_gP(w) U(w) - P(w)g(w)\frac{\partial U}{\partial w}(w)
$$
From the assumptions, this yields the following property~:
$$
L_FP(w) = L_f P(w)   -  \lambda \left |P(w)g(w)\right |^2 \leq -Q
$$
From Proposition \ref{Prop_LyapSuffGlobStab}, it implies the result.
\end{proof}
Following \cite{AndrieuJayawardhanaPraly_TAC_TransExpStab}, it is possible to slightly relax these assumptions by introducing a scaling factor $\alpha(w)$ which multiply $g$ and by rewriting these assumption accordingly.

\section{Conclusion and final remark}

In this note, it has been shown how first order approximation study may lead to the construction of Lyapunov function that characterizes the local exponential stability of a transverse manifold. 
In the context of stabilization of an equilibrium point, a global Lyapunov function may be constructed from first order approximation.
In this case, one has to consider the Riemaniann length to the origin as a Lyapunov function.

An interesting question is to consider the problem of global stability property for a transverse manifold.
However, as it can be shown in this simple example, some problematic effect can show up.
Consider the following planar system defined on $\RR^2$:
\begin{equation}\label{eq_Example}
\dot e = -\phi (x)e \ , \ \dot x = \mu_x x \ , \ \phi(x) = \lambda + x \sin(x)\ .
\end{equation}
It can  be checked, that its solutions are defined for all $t$ in $\RR$ as~:
$$
E(e_0,x_0,t) = \exp\left(-\lambda t+\frac{\cos(1)-\cos(e^{\mu_x t}x_0)}{\mu_x}\right) e_0 \ ,\ X(e_0,x_0,t) = e^{\mu_x t} x_0\ .
$$
This implies that the manifold $\{(e,x), e=0\}$ is locally exponential stable and globally attractive uniformly in $x$.
Indeed, we have for all $(e_0,x_0)$ in $\RR$~:
$$
|E((e_0,x_0),t)|\leq \exp\left(\frac{\cos(1)+1}{\mu_x}\right) \exp(-\lambda t)|e_0|\ .
$$
However, note that if we consider the transversally linear system, we have~:
$$
\dot {\overparen{\begin{bmatrix}
\dE\\ \dX
\end{bmatrix}}} = \begin{bmatrix}
\phi(e^{\mu_x t}x_0) & \phi'(e^{\mu_x t}x_0)E(e_0,x_0,t)\\0 & \mu_x
\end{bmatrix}\begin{bmatrix}
\dE\\ \dX
\end{bmatrix} \ ,
$$
which gives (with $\dE(t)=\dE(\de_0,\dx_0,e_0,x_0,t)$)~:
\begin{align*}
\dE(t) &=\exp\left(\int_0^t \phi(e^{\mu_x s}x_0)ds\right)\de_0 +  \int_0^t \exp\left(\int_s^t \phi(e^{\mu_x \nu}x_0)d\nu\right) \phi'(e^{\mu_x \nu}x_0)\\&
\qquad\qquad\qquad\qquad\qquad\qquad\qquad\qquad\qquad
\times E(w_0,s)e^{\mu_x s}\dx_0ds\ ,\\
&=  \exp\left(\int_0^t \phi(e^{\mu_x s}x_0)ds\right)\left[\de_0 + \int_0^t \phi'(e^{\mu_x s}x_0)e^{\mu_x s}e_0\dx_0 ds\right]\ .
\end{align*}
Hence, this yields $x_0\neq 0$,
\begin{align*}
\dE(t) 
&=  \exp\left(\int_0^t \phi(e^{\mu_x s}x_0)ds\right)\left[\de_0 + \frac{\phi(e^{\mu_x t}x_0)- \phi(x_0)}{\mu_x}\frac{e_0\dx_0}{x_0}\right]\ .
\end{align*}
With $\phi$ previously defined
it gives~:
\begin{multline*}
\dE( t) =  \exp\left(\frac{\cos(x_0)-\cos(e^{\mu_x t}x_0)}{\mu_x}\right) \\\times\left[e^{-\lambda t}\de_0 + \frac{e^{(\mu_x-\lambda) t}\sin(e^{\mu_x t}x_0)- \sin(x_0)e^{-\lambda t}}{\mu_x}e_0\dx_0\right]\ .
\end{multline*}
It can be checked that $\dE( t)$ doesn't converge to zero if $\lambda < \mu_x$ if for instance
 $e_0=1$, $x_0=1$, $\dx_0=1$.
From this remarks, this implies that the study of the linearized system have to be taken with care.
Indeed, this implies that the \factun given in the introduction  is no longer valid in this context.
More precisely, exponential convergence to the origin of the $e$ dynamics, doesn't imply that the $\de$ component of the linearized system along the solutions converges to zero.

In \cite{AndrieuJayawardhanaPraly_GlobTransExpStab}, it has been shown that when the convergence rate to the manifold is larger then the expansion rate in the manifold, \factun may hold.
In this case, it is possible to construct a Lyapunov function based on first order approximation.

Finally, the construction of a matrix function $P$ which satisfies equations (\ref{eq_TensorDerivativeTrans}), (\ref{eq_dPdtGlob}) or  (\ref{eq_dPdtGlobStabilization}) is a crucial step in order to make this framework interesting from a practical point of view.
Preliminary results aiming at solving a differential Riccati equation (as the one given in (\ref{eq_dPdtGlobStabilization})) are given in \cite{SanfelicePraly_CDC_15}.
Backstepping based approaches is also a possible research line (see \cite{ZamaniTabuada_TAC_11BacksteppingIncrStab} or \cite{ZamanietAl_SCL_2013backstepping}).
Finally, a method following a numerical approximation of the partial differential equation should also be considered.

\bibliographystyle{plain}
\bibliography{BibVA}

\begin{thebibliography}{10}

\bibitem{AndrieuJayawardhanaPraly_CDC_13}
V.~Andrieu, B.~Jayawardhana, and L.~Praly.
\newblock On transverse exponential stability and its use in incremental
  stability, observer and synchronization.
\newblock In {\em Proc. of the 52nd IEEE Conference on Decision and Control},
  2013.

\bibitem{AndrieuJayawardhanaPraly_GlobTransExpStab}
V.~Andrieu, B.~Jayawardhana, and L.~Praly.
\newblock {Globally transverse exponential stability}.
\newblock Technical report, 2015.

\bibitem{AndrieuJayawardhanaPraly_TAC_TransExpStab}
V.~Andrieu, B.~Jayawardhana, and L.~Praly.
\newblock {Transverse exponential stability and applications}.
\newblock {\em Submitted to IEEE Transactions on Automatic Control}, 2015.

\bibitem{AndrieuJayawardhanaTarbouriech_CDC_Synchro}
V.~Andrieu, B.~Jayawardhana, and S.~Tarbouriech.
\newblock {Necessary and sufficient condition for local exponential
  synchronization of nonlinear systems.}
\newblock In {\em Proc. of the 54th IEEE Conference on Decision and Control},
  2015.

\bibitem{Angeli_TAC_02_lyapIncStab}
D.~Angeli.
\newblock A lyapunov approach to incremental stability properties.
\newblock {\em Automatic Control, IEEE Transactions on}, 47(3):410--421, 2002.

\bibitem{Brezis_Book_73}
H.~Brezis.
\newblock {\em {Op{\'e}rateur maximaux monotones et semi-groupes de
  contractions dans les espaces de Hilbert}}, volume~5.
\newblock Mathematics Studies, 1973.

\bibitem{ForniSepculchre_TAC_2014}
F.~Forni and R.~Sepulchre.
\newblock A differential lyapunov framework for contraction analysis.
\newblock {\em Automatic Control, IEEE Transactions on}, 59(3):614--628, March
  2014.

\bibitem{ForniSepulchreVFSchaft_13_CDC_differentialPass}
F.~Forni, R.~Sepulchre, and A.~J. Van Der~Schaft.
\newblock On differential passivity of physical systems.
\newblock In {\em Decision and Control (CDC), 2013 IEEE 52nd Annual Conference
  on}, pages 6580--6585. IEEE, 2013.

\bibitem{GauthierKupka_Book_01}
J.-P. Gauthier and I.~Kupka.
\newblock {\em {Deterministic observation theory and applications}}.
\newblock Cambridge University Press, 2001.

\bibitem{GruneSontagWirth_SCL_99asymptotic}
L.~Gr{\"u}ne, E.~D. Sontag, and F.~R. Wirth.
\newblock Asymptotic stability equals exponential stability, and iss equals
  finite energy gain—if you twist your eyes.
\newblock {\em Systems \& Control Letters}, 38(2):127--134, 1999.

\bibitem{Hartman_Book_64}
P.~Hartman.
\newblock {\em {Ordinary differential equations}}.
\newblock Wiley, 1964.

\bibitem{IsacNemeth_Book_08}
George Isac and S{\'a}ndor~Zolt{\'a}n N{\'e}meth.
\newblock {\em Scalar and asymptotic scalar derivatives: theory and
  applications}, volume~13.
\newblock Springer, 2008.

\bibitem{Isidori_Book_89}
A.~Isidori.
\newblock {\em {Nonlinear control systems: an introduction}}.
\newblock Springer-Verlag New York, Inc. New York, NY, USA, 1989.

\bibitem{Khalil_Book_02}
H.K. Khalil.
\newblock {\em Nonlinear Systems}.
\newblock Prentice-Hall, 3rd edition, 2002.

\bibitem{Kurzweil_AMST_56}
J.~Kurzweil.
\newblock {On the inversion of Lyapunov second theorem on stability of motion}.
\newblock {\em Ann. Math. Soc. Trans. Ser.}, 2(24):19--77, 1956.

\bibitem{Lewis_49_AJM_MetPropDiffEq}
D.~C. Lewis.
\newblock Metric properties of differential equations.
\newblock {\em American Journal of Mathematics}, 71(2):294--312, 1949.

\bibitem{Lyapunov_92_IJC_general}
A.~M. Lyapunov.
\newblock The general problem of the stability of motion.
\newblock {\em International Journal of Control}, 55(3):531--534, 1992.

\bibitem{Massera_AnnalMath_49}
J.L. Massera.
\newblock {On Liapunoff s condition of stability}.
\newblock {\em Annals of Mathematics}, (50):705--721, 1949.

\bibitem{Massera_AnnalMath_56}
J.L. Massera.
\newblock {Contributions to stability theory}.
\newblock {\em Annals of Mathematics}, pages 182--206, 1956.

\bibitem{Praly_Poly_08}
L.~Praly.
\newblock {\em Fonctions de Lyapunov, Stabilit\'e et Stabilisation}.
\newblock Ecole Nationale Sup\'erieure des Mines de Paris, 2008.

\bibitem{Reich_Book_05_NLSemGrp}
S.~Reich.
\newblock {\em Nonlinear semigroups, fixed points, and geometry of domains in
  Banach spaces}.
\newblock Imperial College Press, 2005.

\bibitem{RufferEtAl_SCL_13ConvSystIncStab}
B.~S. R{\"u}ffer, N.~van~de Wouw, and M.~Mueller.
\newblock Convergent systems vs. incremental stability.
\newblock {\em Systems \& Control Letters}, 62(3):277--285, 2013.

\bibitem{SanfelicePraly_TAC_12}
R.G. Sanfelice and L.~Praly.
\newblock Convergence of nonlinear observers on $\mathbb{R}^{n}$ with a
  riemannian metric (part i).
\newblock {\em Automatic Control, IEEE Transactions on}, 57(7):1709--1722,
  2012.

\bibitem{SanfelicePraly_CDC_15}
R.G. Sanfelice and L.~Praly.
\newblock Solution of a riccati equation for the design of an observer
  contracting a riemannian distance.
\newblock In {\em Decision and Control, 2015 Proceedings of the 54th IEEE
  Conference on}. IEEE, 2015.

\bibitem{SepulchreJankovic_Book_97}
R.~Sepulchre, M.~Jankovi{\'c}, and P.~V. Kokotovi{\'c}.
\newblock {\em Constructive nonlinear control}.
\newblock Communications and Control Engineering Series. Springer-Verlag, 1997.

\bibitem{TeelPraly_ESAIM_00}
A.R. Teel and L.~Praly.
\newblock {A smooth Lyapunov function from a class-\$$\{$$\backslash$ mathcal
  $\{$KL$\}$$\}$\$ estimate involving two positive semidefinite functions}.
\newblock {\em ESAIM: Control Optim. Calc. Var.}, 5:313--367, 2000.

\bibitem{ZamaniTabuada_TAC_11BacksteppingIncrStab}
M.~Zamani and P.~Tabuada.
\newblock Backstepping design for incremental stability.
\newblock {\em Automatic Control, IEEE Transactions on}, 56(9):2184--2189,
  2011.

\bibitem{ZamanietAl_SCL_2013backstepping}
M.~Zamani, N.~van~de Wouw, and R.~Majumdar.
\newblock Backstepping controller synthesis and characterizations of
  incremental stability.
\newblock {\em Systems \& Control Letters}, 62(10):949--962, 2013.

\end{thebibliography}

\end{document}